\documentclass[a4paper,11pt]{amsart}
\usepackage[english]{babel}
\usepackage{enumerate}

\usepackage{hyperref}

\newtheorem{theorem}{Theorem}[section]
\newtheorem{lm}[theorem]{Lemma}

\theoremstyle{definition}

\theoremstyle{remark}
\newtheorem{remark}[theorem]{Remark}

\usepackage{amssymb}

\newcommand{\NN}{\mathbb{N}}
\newcommand{\RR}{\mathbb{R}}

\newcommand{\ep}{\varepsilon}

\newcommand{\de}{\delta}

\newcommand{\al}{\alpha}
\newcommand{\si}{\sigma}
\newcommand{\ga}{\gamma}

\newcommand{\Om}{\Omega}
\newcommand{\pa}{\partial}
\newcommand{\RN}{\RR^N}
\newcommand{\ovr}{\overline}
\newcommand{\Ga}{\Gamma}
\newcommand{\La}{\Lambda}

\newcommand{\De}{\Delta}

\def\veps{\varepsilon}

\def\NN{\mathbb N}
\def\SS{\mathbb S}
\def\SSN{\SS^{N-1}}

\def\cA{\mathcal{A}}

\def\cH{\mathcal{H}}

\def\RR{\mathbb R}
\def\SS{{\mathbb S}}
\def\RN{{\mathbb R^N}}

\def\Om{\Omega}
\def\De{\Delta}
\def\Ga{\Gamma}
\def\Si{\Sigma}
\def\al{\alpha}

\def\ga{\gamma}
\def\la{\lambda}

\def\om{\omega}
\def\pa{\partial}

\def\veps{\varepsilon}
\def\de{\delta}

\def\ds{\displaystyle}
\def\ovr{\overline}

\newcommand{\fka}{\frak a}
\newcommand{\fkb}{\frak b}

\newcommand{\fkd}{\frak d}
\newcommand{\fkr}{\frak r}
\newcommand{\fkzero}{\frak 0}
\newcommand{\fkR}{\frak R}
\newcommand{\fkH}{\frak H}

\newcommand{\fkL}{\frak L}

\DeclareMathOperator{\dist}{\mathrm{dist}}

\begin{document}
    \title[]{Symmetry and linear stability in Serrin's overdetermined problem
via the stability of the parallel surface problem}

  \date{}

\author[G. Ciraolo]{Giulio Ciraolo}
\address{Dipartimento di Matematica e
Informatica, Universit\`a di Palermo, Via Archirafi 34, 90123, Italy.}
\email{giulio.ciraolo@unipa.it}
\urladdr{http://www.math.unipa.it/~g.ciraolo/}

\author[R. Magnanini]{Rolando Magnanini}
    \address{Dipartimento di Matematica ed Informatica ``U.~Dini'',
Universit\` a di Firenze, viale Morgagni 67/A, 50134 Firenze, Italy.}
    \email{magnanin@math.unifi.it}
    \urladdr{http://web.math.unifi.it/users/magnanin}


\author[V. Vespri]{Vincenzo Vespri}
    \address{Dipartimento di Matematica ed Informatica ``U.~Dini'',
Universit\` a di Firenze, viale Morgagni 67/A, 50134 Firenze, Italy.}
    \email{vespri@math.unifi.it}
    \urladdr{http://web.math.unifi.it/users/vespri}

\begin{abstract}
We consider the solution of the problem
$$
-\De u=f(u) \ \mbox{ and } \  u>0 \ \ \mbox{ in } \ \Om, \ \ u=0 \ \mbox{ on } \ \Ga,
$$
where $\Om$ is a bounded domain in $\RR^N$ with boundary $\Ga$ of class $C^{2,\tau}$, $0<\tau<1$, and $f$ is a locally Lipschitz continuous non-linearity. 
Serrin's celebrated symmetry theorem states that, if the normal derivative $u_\nu$ is constant on $\Ga$, then $\Om$ must be a ball. 
\par
In \cite{CMS2}, it has been conjectured that Serrin's theorem may be obtained {\it by stability} in the following way: first, for a solution $u$ 
prove the estimate
$$
r_e-r_i\le C_\de\,[u]_{\Ga^\de}
$$
for some constant $C_\de$ depending on $\de>0$, where $r_e$ and $r_i$ are the radii of a spherical annulus containing $\Ga$, 
$\Ga^\de$ is a surface parallel to $\Ga$ at distance $\de$ and sufficiently close to $\Ga$,
and $[u]_{\Ga^\de}$ is the Lipschitz semi-norm of $u$ on $\Ga^\de$; secondly, if in addition  $u_\nu$ is constant on $\Ga$, show that
$$
[u]_{\Ga^\de}=o(C_\de)\ \mbox{ as } \ \de\to 0^+.
$$
In this paper, we prove that this strategy is successful. 
\par
As a by-product of this method, for $C^{2,\tau}$-regular domains, we also obtain a linear stability estimate for Serrin's symmetry result. Our result is optimal and greatly improves the similar logarithmic-type estimate of \cite{ABR}
and the H\"{o}lder estimate of \cite{CMV} that was restricted to convex domains.
\end{abstract}


\maketitle

\section[Introduction]{Introduction}

 In this paper we establish a connection between two overdetermined boundary value problems,
Serrin's symmetry problem and what we call the {\it parallel surface problem}. As a consequence, 
we obtain optimal stability for the former, thus significantly improving previous results (\cite{ABR}, \cite{CMV}).
\par
Serrin's symmetry problem concerns solutions of elliptic partial differential equations subject to both Dirichlet and Neumann boundary conditions. Let $\Om\subset\RN$ be a bounded domain with sufficiently smooth boundary $\Ga$ (say $C^{2,\tau}$, $0<\tau<1$). As shown in one of his seminal papers, \cite{Se},
if the following problem 
\begin{eqnarray}
&\Delta u + f(u) = 0 \ \mbox{ and } \
u>0 \ \mbox{ in } \ \Om, \ u=0 \ \mbox{ on } \ \Ga,
\label{serrin1}
\\
&u_\nu=\fka \ \mbox{ on } \ \Ga,
\label{serrin2}
\end{eqnarray}
admits a solution for a given positive constant $\fka$, then $\Om$ must be a ball. Here,  $f:[0,+\infty)\to\RR$ is a locally Lipschitz continuous function and $\nu$ is the {\it inward} unit vector field to $\Ga$.
Generalizations of these result
are innumerable and we just mention \cite{BCN}, \cite{BNV}, \cite{GL}, \cite{GNN}, \cite{Re}.

\par
The parallel surface problem concerns solutions of \eqref{serrin1} with a level surface {\it parallel} to $\Ga$,
that is to say the solution $u$ of \eqref{serrin1} is required to be constant at a fixed distance from $\Ga$:
\begin{equation}
\label{parallelcondition}
u=\frak b \ \mbox{ on } \ \Ga^\de.
\end{equation}
Here, 
\begin{equation} \label{Gamma delta}
\Ga^\de=\{ x\in\Om: d_\Om(x)=\de\} \ \ \mbox{ with } \ \ d_\Om(x)=\min_{y\notin\Om} |x-y|, \ x\in\RN,
\end{equation}
is a surface parallel to $\Ga$, and $\fkb$ and $\de$ are positive (sufficiently small) constants.
Under sufficient conditions on $\Ga^\de$, also in this case, if a solution of \eqref{serrin1} and \eqref{parallelcondition} exists, then $\Om$ must be a ball (see \cite{MSaihp, MSmmas, Sh, CMS1, GGS}). In the sequel, we will occasionally refer to problem \eqref{serrin1}, \eqref{parallelcondition} as the {\it parallel surface problem}.
\par
A condition like \eqref{parallelcondition} was considered in \cite{MSaihp, MSmmas} in connection with
{\it time-invariant} level surfaces of a solution $v$ of the 
non-linear equation
$$
v_t=\De\phi(v) \ \mbox{ in } \ \Om\times(0,\infty)
$$
subject to the initial and boundary conditions
$$
v=1 \ \mbox{ on } \ \Om\times\{0\} \ \mbox{ and } \ v=0  \ \mbox{ on } \ \pa\Om\times(0,\infty);
$$
here, $\phi$ is a $C^2$-smooth non-linearity with $\phi(0)=0$ and $\phi'$ bounded
from below and above by two positive constants (hence, we are dealing with a non-degenerate fast-diffusion equaton).
A (spatial) level surface $\Si\subset\Om$ of $v$  is time-invariant
if $v$ is constant on $\Si$ for each fixed time $t>0$.
\par
It is proved in \cite{MSaihp} that for $x\in\ovr{\Om}$
$$
4t\,\int_0^{v(x,t)}\frac{\phi'(\eta)}{1-\eta}\,d\xi\longrightarrow d_\Om(x)^2 \ \mbox{ as } \ t\to 0^+,
$$
uniformly on compact subsets of $\Om$. Hence, if $\Si$ is time-invariant, then it has to be parallel
to $\Ga$ at some distance $\de$.
Also, it is not difficult to show that the function
$$
u(x)=\int_0^\infty \phi(v(x,t))\,dt, \ x\in\ovr{\Om},
$$
satisfies \eqref{serrin1}, with $f(u)=1$ and, being $\Si$ time-invariant, $u$ satisfies \eqref{parallelcondition}, with $\Ga^\de=\Si$, for some positive constant $\fkb$.
As a consequence, $\Om$ is a ball and, in this situation, the time-invariant surface $\Si$ turns out to be
a sphere.
\par
Condition \eqref{parallelcondition} can also be re-interpreted to give a connection to transnormal and
isoparametric functions and surfaces. We recall that, in differential geometry, a function $u$ is {\it transnormal} in $\Om$ if it is a solution of the equation
\begin{equation}
\label{transnormal}
|Du|^2=g(u) \ \mbox{ in } \ \Om,
\end{equation}
for some suitably smooth function $g:\RR\to (0,\infty)$; the level surfaces of $u$ are called
{\it transnormal surfaces}.  A transnormal function that also satisfies the first equation in \eqref{serrin1}
is called an {\it isoparametric function} and its level surfaces are {\it isoparametric surfaces}. Isoparametric surfaces in the euclidean space can only be (portions of) spheres, spherical cylinders or hyperplanes (\cite{L-C}, \cite{Sg});
a list of essential references about transnormal and isoparametric functions and their properties
in other spaces includes \cite{Bo}, \cite{Ca}, \cite{Mi}, \cite{Wa}.
\par
Now, notice that the (viscosity) solution $u$ of \eqref{transnormal} such that $u=0$ on $\Ga$
takes the form $u(x)=h(\dist(x,\pa\Om))$, where $h$ is defined by
$$
\int_0^{h(t)}\frac{ds}{\sqrt{g(s)}}=t, \ \ t\ge 0.
$$
It is then clear that a solution of \eqref{serrin1} satisfies \eqref{parallelcondition} if and only if $u=\fkb$ on $\{x\in\Om: w(x)=h(\de)\}$. Thus, we can claim that {\it if a solution of \eqref{serrin1} has a level surface that is also a level surface of a transnormal function, then $\Om$ must be a ball.}
\par
Both problems \eqref{serrin1}, \eqref{serrin2} and \eqref{serrin1}, \eqref{parallelcondition} have at least one feature in common: the proof of symmetry relies on the {\it method of moving planes}, a refinement, designed by J. Serrin, of a previous idea of V.I. Aleksandrov's \cite{Al}. The evident similarity between the two problems arouses a natural question:
{\it to obtain the symmetry of $\Om$, is condition \eqref{serrin2} weaker or stronger than \eqref{parallelcondition}?}
\par
As noticed in \cite{CMS2} and \cite{CM}, condition
\eqref{parallelcondition} seems to be weaker than \eqref{serrin2}, in the sense clarified hereafter. As \eqref{parallelcondition} does not imply \eqref{serrin2}, the latter can be seen as the limit of a sequence of conditions of type \eqref{parallelcondition} with $\fkb=\fkb_n$ and $\de=\de_n$ and
$\fkb_n$ and $\de_n$ vanishing as $n\to\infty$. As \eqref{serrin2} does not imply \eqref{parallelcondition} either, nonetheless  the oscillation on a surface parallel to $\Ga$ of a solution of \eqref{serrin1}, \eqref{serrin2} becomes smaller than usual, the closer the surface is to $\Ga$.
A way to quantitavely express this fact is to consider the Lipschitz seminorm 
\begin{equation*}
 [u]_{\Ga^\de} =  \sup_{{\substack{x,y \in\Ga^\de, \\ x\neq y}}} \frac{|u(x)-u(y)|}{|x-y|}\,,
\end{equation*}
that controls the oscillation of $u$ on $\Ga^\de$: it is not difficult to show by a Taylor-expansion argument (see the proof of Theorem \ref{thm Serrin by stability}) that,
if $u\in C^{2,\tau}(\ovr{\Om})$, $0<\tau<1$, satisfies \eqref{serrin1}, \eqref{serrin2}, then
\begin{equation}
\label{seminorm-tau}
[u]_{\Ga^\de}=O(\de^{1+\tau})  \ \mbox{ as } \ \de\to 0.
\end{equation}


\par
This remark suggests the possibility that Serrin's symmetry result may be obtained {\it by stability} by the following strategy: (i) for the solution $u$ of \eqref{serrin1} prove that,
for some constant $C_\de$ depending on $\de$,  
an estimate of type
\begin{equation}
\label{estimate}
 r_e-r_i\le C_\de\, [u]_{\Ga^\de}
\end{equation}
holds for any sufficiently small $\de>0$, where $r_e$ and $r_i$ are the radii of a spherical annulus centered at some point $\fkzero$, $\{x\in\RN: r_i<|x-\fkzero|<r_e\}$, containing $\Ga$ --- this means that $\Om$ is nearly a ball, if $u$ does not oscillate too much on $\Ga^\de$; 
(ii)
if in addition  $u$ satisfies \eqref{serrin2}, that is $u_\nu$ is constant on $\Ga$, show that
$$
\ [u]_{\Ga^\de}=o(C_\de)\ \mbox{ as } \ \de\to 0^+.
$$
The spherical symmetry of $\Ga$ then will follow by choosing $\de$ arbitrarily small. 
\par
In \cite{CMS2}, the first two authors of this paper proved an estimate of type \eqref{estimate}. Unfortunately, that estimate is not sufficient for our aims, since the computed constant $C_\de$ blows up exponentially as $\de\to 0^+$. Nevertheless, in the case examined in
\cite{CM}, we showed that our stategy is successful for the very special class of ellipses.
\par
In this paper we shall extend the efficacy of our strategy to the class of $C^{2,\tau}$-smooth domains.
The crucial step in this direction  is Theorem \ref{thm stability dependence on s}, where we considerably improve inequality \eqref{estimate} by
showing that it holds with a constant $C_\de$ that is a $O(\de^{-1})$ as $\de\to 0^+$. Thus, \eqref{seminorm-tau} will imply that $r_i=r_e$, that is $\Om$ is a ball (see Theorem \ref{thm Serrin by stability}). By a little more effort, in Theorem \ref{thm stability by stability} we will prove that
$$
r_e-r_i\le C\,[u_\nu]_\Ga,
$$
for some positive constant $C$. This inequality enhances to optimality two previous results, both
also based on the method of moving planes. In fact, it 
improves the logarithmic stability obtained in \cite{ABR} for $C^{2,\tau}$-regular domains and extends the linear stability obtained for convex domains in \cite{CMV}. We notice that, in our inequality the seminorm $[u_\nu]_\Ga$ replaces the deviation in the $C^1(\Ga)$-norm 
 of the function $u_\nu$ from a given constant,  considered in  \cite{ABR}. Moreover, the inequality also
improves \cite{BNST}[Theorem 1.2], where a H\"older-type estimate was obtained, for the case in which $f(u)=-N$, by means of integral identities.
\par
The outline of the proof of Theorem \ref{thm stability dependence on s} will be recalled in Section 2: it is the same as that of \cite[Theorem 4.1]{CMS2}, that relies on ideas introduced in \cite{ABR}, and the use of {\it Harnack's} and
{\it Carleson's (or the boundary Harnack's) inequalities}. In Section 3 --- the heart of this paper --- by the careful use of refined versions of those inequalities (see \cite{Ba}, \cite{BCN}), we prove the necessary lemmas that   in Section 4 allow us to obtain
our optimal version of \eqref{estimate}. 
Finally, in Section 5, we present our new linear stability estimate for the radial symmetry in Serrin's problem; it implies symmetry for \eqref{serrin1}, \eqref{serrin2}: thus, the new strategy is successful. 


\setcounter{equation}{0}
\setcounter{theorem}{0}

\section{A path to stability: \\
the quantitative method of moving planes} 
\label{sec:preliminaries}
In this Section we introduce some notation and we review the quantitative study of the method of moving planes as carried out in \cite{ABR} and \cite{CMS2} (see also \cite{CV}).

Let $\Om$ be a bounded domain in $\RR^N$ $(N\geq 2)$ and $\Ga$ be its boundary; we shall denote the diameter of $\Om$ by $\fkd_\Om$. The distance $d_\Om$ defined in \eqref{Gamma delta} is always Lipschitz continuous on $\ovr{\Om}$ and of class $C^{2,\tau}$
in a neighborhood of $\Ga$, if this is of class $C^{2,\tau}$, $0<\tau<1$. In fact, under this assumption on $\Ga$, for every $x\in\Ga$ there are balls 
$B\subset\Om$ such that $x\in\pa B$; denote
by $r_x$ the supremum of the radii of such balls and set 
$$
\fkr_\Om=\min_{x\in\pa\Om}r_x.
$$ 
We then denote by $\fkR_\Om$ the number obtained by this procedure, where instead the interior ball $B$ is replaced by an exterior one. 
\par
For $\de>0$, let $\Om^\de$ be the {\it parallel set} (to $\Ga$), i.e.
\begin{equation}\label{Om delta}
\Om^\de= \{x \in \Om:\ d_\Om(x)>\de\} \,.
\end{equation}
We know that if $0\le\de<\fkr_\Om$, then each level surface $\Ga^\de$ of $d_\Om$,  as defined by \eqref{Gamma delta}, is of class $C^{2,\tau}$ and will be referred to as
a {\it parallel surface} (to $\Ga$).




The following notations are useful to carry out the method of moving planes and its quantitative version;
for $\om\in\SSN$ and $\mu\in\RR$, we set:
\begin{equation}
\label{definitions}
\begin{array}{lll}
&\pi_{\mu}=\{ x\in\RN: x\cdot\om=\mu\}\ &\mbox{a hyperplane orthogonal to $\om,$}\\
&\cH_\mu=\{ x\in\RN: x\cdot\om>\mu\}\ &\mbox{the half-space \emph{on the right} of $\pi_\mu$,}\\
&A_{\mu}=A\cap\cH_\mu &\mbox{the right-hand cap of a set $A$},\\
&x'=x-2(x\cdot\om-\mu)\,\om\ &\mbox{the reflection of $x$ in $\pi_{\mu},$}\\
&(A_\mu)'=\{x\in\RN:x'\in A_{\mu}\}\ &\mbox{the reflected cap in $\pi_{\mu}$.}
\end{array}
\end{equation}
In the sequel, we will generally use the simplified notation $A'=(A_\mu)'$ every time in which the dependence on $\mu$ is not important.
\par
Set $\La=\sup\{x\cdot\om: x\in \Om\}$; if $\mu<\La$ is close to $\La$, the reflected cap $(\Om_\mu)'$ is contained in $\Om$ (see \cite{Fr}), and hence we can define the number
\begin{equation}\label{m def}
\la=\inf\{\mu: (\Om_{\tilde\mu})'\subset \Om \mbox{ for all } \tilde\mu\in(\mu,\La)\}.
\end{equation}
Thus, at least one of the following two cases occurs (\cite{Se},\cite{Fr}):
 \begin{enumerate}
\item[(S1)]
$\Om'=(\Om_{\la})'$ is internally tangent to $\pa \Om$ at some
point $p'\in\pa \Om'\setminus\pi_{\la}$, which is the reflection in $\pi_\la$ of a point $p\in \pa \Om_\la\setminus\pi_{\la}$;
\item[(S2)]
$\pi_{\la}$ is orthogonal to $\pa \Om$ at some point $q\in\pa \Om\cap\pi_{\la}$.
\end{enumerate}
In the sequel, $\pi_\la$ and $\Om_\la$ will be referred to as the {\it critical hyperplane} and the {\it critical cap} (in the direction $\om$), respectively. Corresponding to the points $p$ and $q$, we will also consider
the points $p^\de=p+\de\,\nu(p)$ and $q^\de=q+\de\,\nu(q)$ for $0<\de< \fkr_\Om$; notice that
$\Ga^\de=\{p^\de:p\in\Ga\}$. 
\par
Let $\fkr\in (0, \fkr_\Om)$. From now on $G$ will denote the parallel set 
$$
G=\{ x\in\Om: d_\Om(x)>\fkr\}.
$$
In Section 4, we shall choose $\fkr$ appropriately. Also, to simplify notations, by $P$ and $Q$ we shall denote $p^{\fkr}$ and $q^{\fkr}$, respectively --- two points on $\pa G$ that will be frequently used.
\par
Now, the function $w$ defined by
\begin{equation}\label{w def}
w(x)= u(x')-u(x),\quad x\in \Om_\la,
\end{equation}
satisfies
\begin{equation*} \label{eq w semilinear}
\De w+c(x)\,w=0 \  \mbox{ in } \Om_{\la},
\end{equation*}
where for $x\in\Om_{\la}$
\begin{equation*}
\label{defc}
c(x)=\left\{
\begin{array}{lll}
\ds\frac{f(u(x'))-f(u(x))}{u(x^\la)-u(x)} &\mbox{ if } u(x')\not= u(x),\\
\ds 0 &\mbox{ if } u(x') = u(x).
\end{array}
\right.
\end{equation*}
Notice that $c(x)$ is bounded by the Lipschitz constant $\fkL$ of $f$ in the interval
$[0,\max\limits_{\ovr{\Om}}u].$ 
\par
By an argument introduced in \cite[Theorem 2]{Se} and improved in \cite{BNV} (see also \cite{Fr}), 
we can assume that $w\ge 0$ in $\Om_\la$ and hence, by the strong maximum principle
applied to the inequality $\De w-c^-(x)\, w\le 0$ with $c^-(x)=\max[-c(x),0]$, we can suppose that
$$
w>0 \ \mbox{ in } \ \Om_\la.
$$
\par
One ingredient in our estimates of Section 3 will be {\it Harnack's inequality}:  thanks to this result, for
fixed $a\in (0,1)$, $w$ satisfies the inequality
\begin{equation} \label{harnack}
\sup_{B_{ar}} w \leq \frak{H}_a \inf_{B_{ar}} w ,
\end{equation}
for any ball $B_r \subset \Om_\la$ (see \cite[Theorem 8.20]{GT}); the Harnack constant $\frak{H}_a$ can be bounded by the power 
$\sqrt{N} + \sqrt{r \fkL}$ of a constant 
only depending on $N$ and $a$ (see \cite{GT}). For instance, if $c(x) \equiv 0$, by the explicit Poisson's representation formula for harmonic functions, we have that
\begin{equation*}
\sup_{B_{ar}} w \leq \left( \frac{1+a}{1-a} \right)^N \inf_{B_{ar}} w ,
\end{equation*}
for any $B_r \subset \Om_\la$  (see \cite{GT} and \cite{DBGV}).

Now, we review the quantitative study of the method of moving planes established in \cite{CMS2}, partly
based on the work in \cite{ABR}. As already mentioned in the Introduction, the stability of the radial configuration for problem \eqref{serrin1}, \eqref{parallelcondition} is obtained in \eqref{estimate} in terms of the Lipschitz seminorm on parallel surfaces to $\Ga$.

For a fixed direction $\om$ we consider the critical positions and the corresponding points $p$ and $q$, as detailed in (S1) and (S2). As shown in \cite{MSmmas}, the method of moving planes can be applied to $G$ instead of $\Om$, and the tangency points of cases (S1) and (S2) are $P$ and $Q$, 
respectively. It is clear that if an estimate like \eqref{estimate} holds for $G$, then the same holds for $\Om$, since the difference of the radii does not change. 
\par
The procedure to obtain \eqref{estimate} is quite delicate. For the reader's convenience, we give an outline of it, in which we identify $8$ salient steps.

\begin{enumerate}[(i)]
\item
Following the proof of \cite[Theorem 3.3]{CMS2}, we show that the values of $w(P)/\dist(P,\pi_\la)$, in case (S1), and of the partial derivative $w_\om(Q)$, in case (S2), are bounded by some constant times $[u]_{\pa G}$. 
    
\item By Harnack's inequality, the smallness obtained in (i) at the points $P$ and $Q$ propagates to any point in $G_\la$ sufficiently far from $\pa G_\la$ (see \cite[Lemma 3.1]{CMS2}).
    
\item By using Carleson's inequality, the estimation obtained in the previous step extends to any point in the cap $G_\la$ (see \cite[Lemma 3.1]{CMS2}), thus obtaining the inequality
 \begin{equation*}
 \|w\|_{L^\infty(G_\la)} \leq C [u]_{\pa G}.
\end{equation*}
Here, the key remark is that $C$ only depends on $N$, $\fkr$, the diameter and the $C^2$-regularity of $G$, but {\it does not} depend on the particular direction $\om$ chosen.

\item The union of $G\cap\ovr{\cH_\la}$ with its reflection in $\pi_\la$ defines a set $X$, symmetric in the direction $\om$, that approximates $G$, since the smallness of $w$ bounds that of $u$ on $\pa X$.   

\item Since $u$ is the solution of \eqref{serrin1}, then $u(x)$ grows linearly with $d_\Om(x)$, when $x$ moves inside $\Om$ from $\Ga$; this implies that $u$ can not be too small on $\pa G=\Ga^{\fkr}$.

\item By using both steps (iv) and (v), we find that the distance of every point in $\pa X$ from $\pa G$ is not greater than some constant times $[u]_{\pa G}$ (\cite[Lemma 3.4]{CMS2}). This means that $X$ {\it fits well} $G$, in the sense that $X$ contains the parallel set $G^\sigma$ (related to $G$) for some positive (small) number $\sigma$ controlled by  $[u]_{\pa G}$ (\cite[Theorem 3.5]{CMS2}). This fact is what we call a {\it quantitative approximate symmetry} of $G$ in the direction $\om$.

\item An approximate center of symmetry $\fkzero$ is then determined as the intersection of $N$ mutually orthogonal critical hyperplanes. As shown in \cite[Proposition 6]{ABR}, the distance between 
$\fkzero$ and any other critical hyperplane can be uniformly bounded in terms of the parameter $\sigma$ in item (vi) and hence of $[u]_{\pa G}$.

\item The point $\fkzero$ is finally chosen as the center of the spherical annulus $\{x\in\RN: r_i<|x-\fkzero|<r_e\}$ and the estimate \eqref{estimate} follows from \cite[Proposition 7]{ABR}.
\end{enumerate}

Based on this plan, to improve \eqref{estimate}, it is sufficient to work on the estimates in step (ii). This will be done in Section 3, by refining our use of Harnack's and Carleson's inequalities. As a matter of fact, in \cite{CMS2} we merely 
used a standard application of Harnack's inequality, by constructing a Harnack's chain of balls of suitably chosen fixed radius $\fkr$. This strategy only yields an exponential dependence on $\fkr^{-1}$ 
of the constant in \eqref{estimate}. In \cite{CMV}, we improved these estimates by chosing a chain of balls with radii that decay linearly when the balls approach $\Ga$; however, this could be done only  when $\Om$ is convex (or little more) and leads to a H\"{o}lder type dependency on $\fkr^{-1}$ of the constant in \eqref{estimate}.
\par
In Section 3 instead, we use the following plan, that we sketch for the case $\om=e_1$ and $\la=0$;
$p$ and $q$ are the points defined in (S1) and (S2).
\par
We fix $\fkr=\fkr_\Om/4$, so that $G=\Om^{\fkr_\Om/4}$. For any $0<\de<\fkr_\Om/4$, the values of $w(x)/x_1$ at the points $p^\de$ and $P=p^{\fkr}$ can be compared in the following way
\begin{equation*}
    w(P)/P_1 \leq C\,\de^{-1}\, w(p^{\de})/p^{\de}_1,
\end{equation*}
where $C$ is a constant not depending on $\de$. Correspondingly, we prove that
\begin{equation*}
    w_{x_1}(Q) \leq C_G\,\de^{-1}\, w_{x_1}(q^{\de}),
\end{equation*}
where $Q=q^\fkr$. Then, by exploiting steps (i) and (iii), we obtain that 
\begin{equation} \label{w leq seminorm sect2}
 w \leq C_G\,\de^{-1}\, [u]_{\Ga^\de} \ \mbox{ on the maximal cap of $G_\la$,} 
\end{equation}
where the constant $C_G$ is the one obtained in step (iii) by letting $\fkr=\fkr_\Om/4$, and hence it does not depends on $\de$.
\par
Once this work is done, steps (iv)--(viii) can be repeated to find the improved approximate symmetry for the parallel set $G$, which clearly implies that for  $\Om$. We underline the fact that the dependence on $\de$ in \eqref{w leq seminorm sect2} is optimal, as \cite{CM} indicates.

\setcounter{equation}{0}
\setcounter{theorem}{0}

\section{Enhanced stability estimates} 
\label{sec: stability estimates}

In this section, we line up the major changes needed to obtain \eqref{estimate}; they only concern step (ii). The following lemma will be useful in the sequel.

\begin{lm} \label{lemma ball ABR}
Let $p$ and $q$ be the points defined in (S1) and (S2), respectively. 
\par
If $B$ is a ball of radius $\fkr_\Om$, contained in $\Om$ and such that $p$  or $q$ belong to $\pa B$,
then the center of $B$ must belong to $\ovr{\Om}_\la$.
\end{lm}
\begin{proof}
The assertion is trivial for case (S2). If case (S1) occurs, without loss of generality, we can assume that $\om=e_1$ and $\la=0$. Since (S1) holds, the reflected point $p'$ lies on $\pa \Om$ and cannot fall inside $B$, since $p\in\pa B$ and $B\subset \Om$. 
\par
Thus,  if $c$ is the center of $B$, we have that $|c-p'|\geq \fkr_\Om=|c-p|$ and hence $|c_1+p_1| \geq |c_1-p_1|$, which implies that $c_1 \geq 0$, being $p_1>0$.
\end{proof}

Our first estimate is a quantitative version of Hopf's lemma, that will be useful to treat both occurrences (S1) and (S2).

\begin{lm} \label{lemma 1 semilinear}
Let $B_R=\{ x\in\RN: |x|<R\}$ and, for $p\in \pa B_R$ and
$s\in (0,R)$, set $p^s=p+s\,\nu(p)=(1-s/R)\, p.$
\par
Let $c\in L^\infty(B_R)$ and suppose that $w\in C^0(B_R \cup \{p\}) \cap C^2(B_R)$ satisfies the conditions:
\begin{equation*}
\label{delta w semilinear}
\Delta w + c(x) w = 0 \ \mbox{ and } \
w \geq 0 \ \mbox{ in } B_R.
\end{equation*}
\par
Then, there is a constant $A=A(N, R, \|c\|_{\infty})$\footnote{See the proof for its expression.} such that
\begin{equation}
\label{w 0 leq semilinear}
w(0) \leq A\,s^{-1} w(p^s) \ \mbox{ for any }  \ s \in (0,R/2).
\end{equation}
Moreover, if $w(p)=0$, then
\begin{equation}
\label{w 0 leq semilinear II}
w(0) \leq A\, w_{\nu} (p).
\end{equation}
\end{lm}

\begin{proof}
We proceed as in the standard proof of Hopf's boundary point lemma.
\par
Notice that $w$ also satisfies
\begin{equation*}
\Delta w - c^-(x) w \leq 0 \ \mbox{ in } \ B_R,
\end{equation*}
where $c^-(x)=\max(-c(x),0)$. Thus, the strong maximum principle implies that $w >0$ in $B_R$ (unless $w\equiv 0$, in which case the conclusion is trivial).
\par
For a fixed $a\in (0,1)$ and some parameter $\al>0$, set
\begin{equation*}
v(x)=\frac{|x|^{-\al}-R^{-\al}}{(a^{-\al}-1) R^{-\al}} \ \mbox{ for } \ a R\le |x|\le R;
\end{equation*}
notice that $v>0$ in $B_R$, $v=0$ on $\pa B_R$ and $v=1$ on $\pa B_{aR}$.
For $a R< |x|< R$ we then compute that
\begin{multline*}
\De v-c^-(x)\,v\ge \frac{\al^2-(N-2)\al -|x|^2 c^-(x)}{(a^{-\al}-1) R^{-\al}}\,|x|^{-\al-2}\\
\ge
\frac{\al^2-(N-2)\al -R^2\| c^-\|}{(a^{-\al}-1) R^{-\al}}\,|x|^{-\al-2}.
\end{multline*}
Hence, we see that
$$
\De v-c^-(x)\,v\ge 0 \quad \textmd{ in } B_R \setminus \overline{B}_{aR},
$$
if we choose
\begin{equation*}
\al =\frac{ N-2+\sqrt{(N-2)^2+4 R^2 \| c^-\|}}{2}.
\end{equation*}
\par
With this choice of $\al$, the function
\begin{equation*}
z=w-\left[\min_{\pa B_{aR}} w\right] v
\end{equation*}
satisfies the inequalities:
\begin{equation*}
\Delta z - c^-(x)\, z \leq 0 \ \textmd{ in } \ B_R \setminus \overline{B}_{a R} \quad \mbox{ and } \quad
z \geq 0 \   \textmd{ on } \ \pa (B_R \setminus B_{a R}).
\end{equation*}
Thus, the maximum principle gives that $z\ge 0$ and hence that
\begin{equation*}
\min_{\pa B_{aR}} w \leq \frac{w(x)}{v(x)},
\end{equation*}
for $x\in B_R \setminus \overline{B}_{a R}$.
\par
Now, choose $x = p^s$ (since we want that $p^s\in B_R \setminus \overline{B}_{a R}$, the constraint $s/R<1-a$ is needed); we thus have that
\begin{equation} \label{1 semilinear}
\min_{\pa B_{aR}} w \leq \frac{a^{-\al}-1}{(1-s/R)^{-\al}-1}\,w(p^s) \le R\,\frac{a^{-\al}-1}{\al}\,
\frac{w(p^s)}{s},
\end{equation}
where the last inequality holds for the convexity of the function $t\mapsto t^{-\al}$.
\par
Harnack's inequality \eqref{harnack} then yields
\begin{equation*}
\label{3 semilinear}
w(0) \leq \sup_{B_{aR}} w \leq \frak{H}_a \inf_{B_{aR}} w \leq R \fkH_a\,\frac{a^{-\al}-1}{\al}\,
\frac{w(p^s)}{s}.
\end{equation*}
Consequently, by chosing $a=1/2$ and setting
$
A=R \frak{H}_{1/2}  (2^{\al}-1)/\al,
$
we readily obtain \eqref{w 0 leq semilinear}.

Finally, if $w(p)=0$, we readily obtain \eqref{w 0 leq semilinear II} from \eqref{w 0 leq semilinear} and by letting $s$ go to zero.
\end{proof}

The following result is crucial to treat the case (S2).

\begin{lm}
\label{lemma 2 semilinear}
Set $B_R^+=\{ x\in B_R: x_1 >0\}$ and $T=\{ x\in\pa B_R^+: x_1=0\}$. For
any point $q\in\pa B_R\cap T$ and $s\in [0,R)$, define
$q^s = q + s \nu(q)=(1-s/R)\,q.$
\par
Let $c\in L^\infty(B_R)$ and suppose $w \in C^2(B_R^+) \cap C^1(B_R^+ \cup T)$ satisfies the conditions:
\begin{equation*}
\Delta w + c(x) w = 0
 \, \mbox{ and } \,  w \geq 0 \ \textmd { in }  \ B_R^+,
\quad w = 0 \ \textmd{ on } \ T.
\end{equation*}\par
Then, there is a constant $A^*=A^*(N, R, \|c\|_{\infty})$ such that
\begin{equation}
\label{bound-wx1}
w_{x_1}(0) \leq A^* s^{-1}\, w_{x_1} (q^s)
\ \mbox{ for any } \ s \in (0,R/2].
\end{equation}
\end{lm}
\begin{proof}
As in Lemma \ref{lemma 1 semilinear}, we can assume that $w>0$ in $B_R^+$. 
\par
Inequality \eqref{bound-wx1} will be the result of a chain of estimates: with this goal, we introduce the half-annulus $\cA^+= B_R^+ \setminus \overline{B^+_\rho}$ and the cube
$$
Q_\si=\{(x_1,\dots, x_N)\in\RN: \, 0<x_1<2\si, \,|x_i|<\si,\, i=2,\dots, N\}.
$$
For the moment, we choose $0<\rho<R$ and $0<\si\le R/\sqrt{N+3}$, that is in such a way that $Q_{\si}\subset B_R^+$; the precise value of $\rho$ will be specified later.
\par
The first estimate of our chain is \eqref{up-bound-sup} below; in order to prove it,  we introduce the auxiliary function
$$
v(x)=[|x|^{-\al}- R^{-\al}]\,x_1 \ \mbox{ for } \ x \in \ovr{A^+}.
$$
It is clear that $v>0$ in $\cA^+$ and $v=0$ on $\pa B^+_R$; also, we can choose $\al>0$ so that $\De v-c^-(x)\, v\ge 0$ in $\cA^+$ ($\al=(N+\sqrt{N^2+4 R^2 \|c^-\|_\infty})/2$ will do).
\par
We then consider the function $w/v$ on $\pa B^+_\rho$: it is surely well-defined, positive and continuous in $\pa B^+_\rho\setminus T$; also,  it can  be extended to be a continuous function up to $T\cap\pa B^+_\rho$ by defining it equal to its limiting values
$
w_{x_1}(x)/(\rho^{-\al}-R^{-\al}|)
$
for $x\in T\cap\pa B^+_\rho$. With this settings, $w/v$ also turns out to be positive on the whole $\pa B^+_\rho$ since, on $T\cap\pa B^+_\rho$, it is positive by a standard application of Hopf lemma.
\par
These remarks tell us that the minimum of $w/v$ on $\pa B^+_\rho $
is well-defined and positive, and hence that the function
$$
z=w-\min_{\pa B^+_\rho}(w/v)\,v
$$
satisfies the inequalities
\begin{equation*}
\Delta z - c^-(x)\, z \leq 0 \ \textmd{ in } \ \cA^+ \quad \mbox{ and } \quad
z \geq 0 \   \textmd{ on } \ \pa \cA^+.
\end{equation*}
Thus, by the maximum principle, $z\ge 0$ on $\ovr{\cA^+}$ and hence
$$
\min_{\pa B^+_\rho}(w/v)\le w(x)/v(x) \ \mbox{for every} \ x\in \ovr{\cA^+}.
$$
\par
For $s<R-\rho$, we then can take  $x=q^s+\ep e_1$, with $\ep>0$ so small that $x\in \cA^+$,
take the limit as $\ep\to 0$ and, since $w(q^s)=v(q^s)=0$, obtain the inequality
$$
\min_{\pa B^+_\rho}(w/v)\le \frac{w_{x_1}(q^s)}{v_{x_1}(q^s)}=
\frac{w_{x_1}(q^s)}{(R-s)^{-\al}-R^{-\al}}.
$$
Again, by the convexity of $t\mapsto t^{-\al}$, we find that
\begin{equation*}
\min_{\pa B^+_\rho}(w/v)\le \frac{R^{\al+1}}{\al\,s}\,w_{x_1}(q^s),
\end{equation*}
and hence it holds that
\begin{equation}
\label{up-bound-sup}
\min_{x\in\pa B^+_\rho}\frac{w(x)}{x_1}\le \frac{R^{\al+1}(\rho^{-\al}-R^{-\al})}{\al\,s}\,w_{x_1}(q^s).
\end{equation}
\par
The second estimate \eqref{w 2rho e1 leq} below shows that, up to a constant, the minimum in \eqref{up-bound-sup} can be 
bounded from below by the value of $w$ at the center of the cube $Q_{\si}$. To do this, we let $y\in\pa B^+_\rho$ be a point at which the minimum in \eqref{up-bound-sup}
 is attained and set
\begin{equation*}
\hat y =(0,y_2,\ldots,y_N)\quad\textmd{and}\quad \bar y=(\si,y_2,\ldots,y_N);
\end{equation*}
notice that $\hat y$ and $y$ coincide when $y_1=0$. 
\par
The ball $B_{\si}(\bar y)$ is tangent to $\pa B_R^+\cap T$ at $\hat y$ and we can choose $\rho$ such that $B_{\si}(\bar y)\subset B_R^+$; thus, by applying Lemma \ref{lemma 1 semilinear}  to $B_{\si}(\bar y)$ 
with $p=\hat y$, $p^s=y$ and $\nu=e_1$,
we obtain that
\begin{equation*}
w(\bar y) \leq A\,w(y)/y_1 \ \mbox{ if } \ y_1>0,
\end{equation*}
and
\begin{equation*}
w(\bar y) \leq A\,w_{x_1} (y) \ \mbox{ if } \ y_1=0, \ \mbox{ being $w(\hat y)=0$. }
\end{equation*}
Thus, we have proved that
\begin{equation*}\label{w bar y leq}
w(\bar y) \leq A \min_{x\in\pa B^+_\rho}\frac{w(x)}{x_1}.
\end{equation*}
Moreover, if we also choose $\rho$ such that $B_\rho(\bar y)\subset B_{2\rho}(\bar y) \subset B_R^+$, since the point $\si e_1\in B_\rho(\bar y)$, Harnack's inequality shows that
\begin{equation*}
w (\si e_1) \leq \frak{H}_{1/2}\, w(\bar y),
\end{equation*}
and hence we obtain that
\begin{equation}\label{w 2rho e1 leq}
w(\si e_1) \leq A\,\frak{H}_{1/2}\, \min_{x\in\pa B^+_\rho}\frac{w(x)}{x_1}.
\end{equation}
\par
To conclude the proof, we use two estimates contained in \cite{BCN} (see also \cite{Ba}).
First, after some rescaling, we can apply \cite[Lemma 2.1]{BCN} to the square $Q_{\si/2}$ and obtain that
\begin{equation}
\label{harnack-up-to-boundary}
w(t\si/2\, e_1)\le t \,C_1\,\max_{\ovr{Q}_{\si/2}} w,
\end{equation}
for every $t\in(0,1)$, where $C_1$ is the constant in \cite[Lemma 2.1]{BCN} that, in our case, only depends
on $N$, $\| c\|_\infty$ and $R$ (by means of $\si$). Thus, since $w(0)=0$, taking the limit as $t\to 0^+$ gives that
\begin{equation}
\label{pa w pa x1 leq C1 M}
w_{x_1}(0)\le\frac{2 C_1}{\si}\,\max_{\ovr{Q}_{\si/2}} w.
\end{equation}
\par
Secondly, we consider the cube $Q_{\si}$ and again after some rescaling, we use the Carleson-type estimate \cite[Theorem 1.3]{BCN} to obtain that
\begin{equation}
\label{carleson}
\max_{\ovr{Q}_{\si/2}} w \leq 2^q\,B\,w (\si\, e_1),
\end{equation}
where, in our case, the constants $B$ and $q$ in \cite[Eq.(1.6)]{BCN} again only depend on $N$, $R$ and $\| c\|_\infty$. Thus, by \eqref{pa w pa x1 leq C1 M} we have that
\begin{equation}\label{pa w pa x1 leq C3 w}
w_{x_1}(0) \leq \frac{2^{q+1} B\,C_1}{\si}\, w (\si e_1).
\end{equation}
\par
Therefore, by applying \eqref{pa w pa x1 leq C3 w}, \eqref{w 2rho e1 leq} and
\eqref{up-bound-sup}, inequality \eqref{bound-wx1} holds with
$$
A^*=2^{q+1} A\,\frak{H}_{1/2}\, B\,C_1\,\frac{(R/\rho)^\al-1}{\al}\,\frac{R}{\si} ,
$$
where the constants $\rho$ and $\si$ can be chosen as specified along the proof.
\end{proof}

For the treatment of case (S1), we must pay attention to the fact that the point of tangency $p$ may be very close to $\pi_\la$ and the interior touching ball at $p$ may not be contained in the cap.  For this reason, we need the following lemma which gives a uniform treatment of all cases occurring when (S1) takes place.

\begin{lm} \label{lemma 3 semilinear}
Let $\xi =\xi_1 e_1$ with $\xi_1>0$ and set
\begin{equation*}
 B_R^+(\xi)=\{x \in \RN:|x-\xi|<R, \, x_1>0\}, \ T=\{x \in \RN:|x-\xi|<R, \, x_1=0\}. \footnote{Notice that $T$ may be the empty set.} 
\end{equation*}
 For $p\in \pa B_R^+(\xi) \setminus T$, define $p^s$ 
as in Lemma \ref{lemma 1 semilinear}.
\par
Let $c$ be essentially bounded on $B_R(\xi)$ and suppose that $w\in C^2(B_R^+(\xi)) \cap C^0(B_R^+(\xi) \cup T)$ satisfies
\begin{equation*}
\Delta w + c(x)\,w = 0 \
 \mbox{ and } \ w \geq 0 \textmd{ in }  B_R^+(\xi),
\quad w = 0 \ \textmd{ on } \ T.
\end{equation*}
\par
Then, there is a constant $A^\#=A^\#(N, R, \|c\|_{\infty})$ such that
\begin{equation}
\label{bound-wx1-s1}
w(\xi)/\xi_1 \leq A^\#\,s^{-1}\, w (p^s)/p_1^s
\ \mbox{ for any } \ s \in (0,R/2].
\end{equation}
\end{lm}

\begin{proof}
We proceed similarly to the proof of Lemma \ref{lemma 2 semilinear}, with some modifications. 
We shall still use the cube $Q_\si$, but we will instead consider the half annulus $\cA^+=B_R^+(\xi) \setminus \overline{B}_\rho^+(\xi)$.  
\par
Next, we change the auxiliary function $v$: 
$$
v(x)=[|x-\xi|^{-\al}- R^{-\al}]\,x_1;
$$
of course $v=0$ on $\pa B_R^+(\xi)$ and we can still choose $\al$ so large that $v$ satisfies the inequality
$\De v-c^-(x)\, v\ge 0$ in $\ovr{\cA^+}$. Thus, the function\footnote{As in the proof of Lemma \ref{lemma 2 semilinear}, we observe that the function $w/v$ can be extended continuously on $T$ and $w/v>0$ on $\pa B_\rho^+(\xi)$.}
$$
z=w-\min_{\pa B^+_\rho}(w/v)\,v
$$
is such that
$
\Delta z - c^-(x)\, z \leq 0
$ 
in $\cA^+$ and $z \geq 0$ on $\pa \cA^+$.
By the maximum principle, we obtain that $z\ge 0$ on $\ovr{\cA^+}$ and hence
\begin{equation*}
\min_{\pa B^+_\rho(\xi)}(w/v)\le w(x)/v(x) \ \mbox{for every} \ x\in \ovr{\cA^+}.
\end{equation*}
Again, by arguing as in the proof of Lemma \ref{lemma 2 semilinear}, we find that
\begin{equation}\label{up-bound-sup II}
\min_{x\in\pa B^+_\rho(\xi)} \frac{w(x)}{x_1}
\le \frac{R^{\al+1}(\rho^{-\al}-R^{-\al})}{\al\,s}\,\frac{w(p^s)}{p_1^s}.
\end{equation}
\par
Now, to conclude the proof we will treat the cases $\xi_1\geq 2\rho$ and $\xi_1\leq2\rho$, separately.
\par
If $\xi_1\geq 2\rho$, the ball $B_{2\rho}(\xi)$ is contained in $B_R^+(\xi)$ and hence, by Harnack's inequality, we have:
$$
\frac{w(\xi)}{\xi_1}\le \frac{w(\xi)}{2\rho}\le \frac{\frak{H}_{1/2}}{2\rho} \min_{\pa B^+_\rho(\xi)} w\le
\frac{\eta+\rho}{2\rho}\,\frak{H}_{1/2} \min_{x\in\pa B^+_\rho(\xi)}\frac{w(x)}{x_1}.
$$
Thus, \eqref{up-bound-sup II} gives that
\begin{equation*}
\frac{w(\xi)}{\xi_1} \le (\eta+\rho)\,\frak{H}_{1/2} \,\frac{R}{\rho}\,\frac{(R/\rho)^{\al}-1}{4\al\,s}\,\frac{w(p^s)}{p_1^s}.
\end{equation*}
\par
If $\xi_1\leq 2\rho$, we repeat the arguments of the last part of the proof of Lemma \ref{lemma 2 semilinear}, with some slight modification. We take a point $y \in \pa B_\rho^+(\xi)$ at which
the minimum in \eqref{up-bound-sup II} is attained
and set $\bar y=(\si,y_2,\ldots,y_N)$, $\hat y=(0, y_2,\dots,y_N)$. We apply Lemma \ref{lemma 1 semilinear} to $B_\si(\bar y)$, with $p=\hat y$, $p^s=y$ and $\nu=e_1$ and, by inspecting the two cases $y_1>0$ and $y_1=0$, we obtain that
\begin{equation*}
w(\bar y) \leq A \min_{x\in\pa B^+_\rho(\xi)}\frac{w(x)}{x_1}.
\end{equation*}
As before, we choose $\rho$ such that $B_\rho(\bar y)\subset B_{2\rho}(\bar y) \subset B_R^+$ and, since $\si e_1\in B_{\rho}(\bar y)$,  by Harnack's inequality we find that
$
w(\si e_1) \leq \frak{H}_{1/2}\, w(\bar y),
$
and hence
\begin{equation}\label{w 4rho e1 leq}
w(\si e_1) \leq \frak{H}_{1/2} \,A \min_{x\in\pa B^+_\rho(\xi)}\frac{w(x)}{x_1}.
\end{equation}
\par
Now, we apply \cite[Theorem 1.3]{BCN} to the cube $Q_\si$ and obtain \eqref{carleson} as before.
Moreover, again we use
\cite[Lemma 2.1]{BCN} in the cube $Q_{\si/2}$; if $\rho$ is sufficiently small, we have that $\xi_1\le 2\rho\le \si/2$ and hence, applying \eqref{harnack-up-to-boundary} with $t=2\xi_1/\si$ gives that
\begin{equation*}
\frac{ w(\xi)}{\xi_1} \leq C_1 \max_{\overline{Q}_{\si/2}} w;
\end{equation*}
therefore, \eqref{carleson} yields:
\begin{equation*}
\frac{ w(\xi)}{\xi_1} \leq 2^q B\,C_1 w (\si e_1).
\end{equation*}
From \eqref{w 4rho e1 leq} and \eqref{up-bound-sup II}, we conclude in this case, as well. 
The constant $A^\#$ can be computed by suitably choosing $\rho$ and $\si$ according to the 
instructions specified in the proof. 
\end{proof}

\setcounter{equation}{0}
\setcounter{theorem}{0}

\section{Approximate symmetry}
In this section, we assist the reader to adapt the theorems obtained in \cite{CMS2} in order to prove our new result of approximate symmetry for $\Om$. First, we prove the analogue of \cite[Theorem 3.3]{CMS2}, that gives an estimate on the symmetry of $\Om$ in a fixed direction.

\begin{theorem} \label{th:w bounded usn}
Let $\Om\subset\RN$ be a bounded domain with boundary $\Ga$ of class $C^{2,\tau}$, $0<\tau<1$,
and set $G=\Om^{\fkr_\Om/4}$. For a unit vector $\om\in\RN$, let $G_\la$ and $\Om_\la$ be the maximal caps in the direction 
$\om$ for $G$ and $\Om$, respectively. 
\par 
Let $u\in C^{2,\tau}(\ovr{\Om})$ be a solution of \eqref{serrin1} and let $w$ be defined by \eqref{w def}.
\par
Then, for every $\de\in (0,\fkr_\Om/8)$, we have that 
\begin{equation}\label{smallness}
w^{\la}\leq C \de^{-1} [u]_{\Ga^\de} \ \mbox{ on (a connected component of)} \ G_\la.
\end{equation}
Here, $C$ is a constant depending on $N$, $\fkL$, $\fkd_\Om$
and the $C^{2,\tau}$-regularity of $\Ga$.
\end{theorem}

\begin{proof}
We point out that $G$ is connected. 
Also, as already done before, we can assume that $\om=e_1$ and  $\la=0$.
\par
Let $p$ and $q$ be the points defined in (S1) and (S2), respectively; $P$ and $Q$ are the points in $\pa G$ already defined.
\par
In what follows, we chose to still denote by $\Om_\la$ and $G_\la$ the connected components of the maximal caps $\Om_\la$ and $G_\la$ that intersect $B_{\fkr_\Om/4}(P)$, if case (S1) occurs, 
and the connected components of $\Om_\la$ and $G_\la$ that intersect $B_{\fkr_\Om/4}(Q)$, if case (S2) occurs.
\par
Lemma \ref{lemma ball ABR} ensures that the interior ball of radius $\fkr_\Om$ touching $\pa \Om$ at $p$ or $q$ has its center in $\overline{\Om}_\la$; hence, $P\in \Om_\la$ and $Q\in \pa \Om_\la \cap \pi_\la$. We then apply \cite[Lemma 4.2]{CMS2} with the following settings: $D_1=G_\la$, $D_2=\Om_\la$, $R=\fkr_\Om/4$, and $z=P$, if case (S1) occurs, and $z=Q$, if case (S2) occurs. Thus, we find that
\begin{equation} \label{w leq in Gm I}
w(x) \leq C\,w(P)/P_1 \ \mbox{ for } \ x \in\ovr{G}_\la,
\end{equation}
and
\begin{equation} \label{w leq in Gm II}
w(x) \leq C\,w_{ x_1}(Q) \ \mbox{ for } \  x \in \ovr{G}_\la,
\end{equation}
respectively. 
Here, the constant $C$ depends only on $N, \fkr_\Om,  \fkL$ and $\fkd_\Om$.

If (S1) occurs, we apply Lemma \ref{lemma 3 semilinear} by letting $R=\fkr_\Om/4$ and $\xi=P$ (this is always possible after a translation in a direction orthogonal to $e_1$), and from \eqref{w leq in Gm I} we  obtain that
\begin{equation} \label{w leq in Gm I II}
w(x) \leq  C\,A^\# \de^{-1} w(p^\de)/p_1^\de \ \mbox{ for } \  x \in \ovr{G}_\la,
\end{equation}
for any $\de\in (0,\fkr_\Om/8)$.
\par
If (S2) occurs, we apply instead Lemma \ref{lemma 2 semilinear} (with $\xi = 0$ and $R=\fkr_\Om/4$) and \eqref{w leq in Gm II}: we find that
\begin{equation} \label{w leq in Gm II II}
w(x) \leq C\,A^* \de^{-1} w_{x_1}(q^\de) \ \mbox{ for } \ x \in \ovr{G}_\la,
\end{equation}
for any $\de \in (0,\fkr_\Om/4)$. 

The rest of the proof runs similarly to that of \cite[Theorem 3.3]{CMS2}, where the estimates of \cite[Lemma 3.2]{CMS2} should be replaced by \eqref{w leq in Gm I II} and \eqref{w leq in Gm II II}. For the reader's convenience, we give a sketch of the proof with the usual settings ($\om=e_1$ and $\la=0$). In particular, we show how to relate $w(p^\de)/p_1^\de$ and $w_{x_1} (q^\de)$ to $[u]_{\Ga_\de}$, which is the main argument of the proof. 

Let us assume that case (S1) occurs. If $p_1^\de \geq \fkr_\Om/2$, since $p^\de$ and its reflection $(p^\de)'$ about $\pi_\la$ lie on $\Ga^\de$, then
\begin{equation*}
w(p^\de) = u((p^\de)')-u(p^\de) \leq \fkd_\Om\,[u]_{\Ga^\de},
\end{equation*}
and hence we easily obtain that
\begin{equation} \label{estim I}
w(p^\de)/p_1^\de\leq 2 \fkd_\Om \fkr_\Om^{-1}\, [u]_{\Ga^\de}.
\end{equation}
If $p_1^\de < \fkr_\Om/2$, then $|p^\de-(p^\de)'|<\fkr_\Om$, then every point of the segment joining $(p^\de)'$ to $p^\de$ is at a distance not greater than $\fkr_\Om$ from some connected component of $\Ga^\de$. The curve $\ga$ obtained by projecting that segment on that component has length bounded by $\widehat C\,|p^\de-(p^\de)'|$, where $\widehat C$ is a constant depending on $\fkr_\Om$ and the regularity of $\Ga^\de$ (and hence on the regularity of $\Ga$ since $\de< \fkr_\Om/8$). 
An application of the mean value theorem to the restriction of $u$ to $\ga$ gives that $u((p^\de)')-u(p^\de)$ can be estimated by the length of $\ga$ times the maximum of the tangential gradient of $u$ on $\Ga^\de$. Thus,
\begin{equation}\label{estim II}
w(p^\de) \leq 2\,\widehat C\, p_1^\de \,[u]_{\Ga^\de}.
\end{equation}
Therefore, \eqref{estim I} and \eqref{estim II} yield the conclusion, if case (S1) is in force.
\par
Case (S2) is simpler. Since $e_1$ belongs to the tangent hyperplane to $\Ga_\de$ at $q^\de$, we readily obtain \eqref{smallness}.
\end{proof}

As outlined in Section \ref{sec:preliminaries}, Theorem \ref{th:w bounded usn} completes steps (i)-(iii) and leads to stability bounds for the symmetry in one direction. Now, we complete steps (iv)-(viii).
\par
As described in steps (iv) and (vi), we define a symmetric open set $X$ and show that $G$ is almost equal to $X$. In order to do that, we need a priori bounds on $u$ from below in terms of the distance function from $\pa G$, as specified in (v). As observed in \cite{ABR} and \cite{CMS2}, such a bound requires a positive lower bound for $u_\nu$ on $\Ga$,
$$
u_\nu\ge\frak c_{\frak 0} \ \mbox{ on } \ \Ga.
$$
If $f(0)>0$, this is guaranteed by Hopf lemma.
If $f(0)\le 0$ instead, such a bound must be introduced as an assumption, as 
it can be realized by considering any (positive) multiple of the first Dirichlet eigenfunction $\phi_1$ for $-\De$. In fact, for any $n\in\NN$ the function $\phi_1/n$ satisfies \eqref{serrin1} with $f(u)=\la_1\,u$, being $\la_1$ the first Dirichlet eigenvalue, and it is clear that,
although $(\phi_1/n)_\nu\to 0$ on $\Ga$ as $n\to\infty$, one cannot expect to derive any
information on the shape of $\Om$. 

The final stability result, step (viii), is obtained by defining an approximate center of symmetry $\fkzero$ as the intersection of $N$ orthogonal hyperplanes as described in step (vii) (see also \cite[Proof of Theorem 1.1]{CMS2}). 
\par
We can now conclude this section with our improved stability estimate on the symmetry of $\Om$.

\begin{theorem} 
\label{thm stability dependence on s}
Let $\Om\subset\RN$ be a bounded domain with boundary $\Ga$ of class $C^{2,\tau}$, $0<\tau<1$.
Let $u\in C^{2,\tau}(\ovr{\Om})$ be a solution of \eqref{serrin1}.
\par 
There exist constants $\veps, C>0$ and $\de_0 \in (0,\fkr_\Om/4)$ such that, if
\begin{equation} 
\label{seminorm-ep}
[u]_{\Ga^{\de_0}} \leq \veps,
\end{equation}
then there are two concentric balls $B_{r_i}$ and $B_{r_e}$ such that
\begin{equation} \label{Bri-Omega-Bre}
B_{r_i}\subset \Om \subset B_{r_e}
\end{equation}
and
\begin{equation} \label{stability t}
r_e-r_i\le C\,\de^{-1}\, [u]_{\Ga^{\de}},
\end{equation}
for any $\de \in (0, \de_0]$.

The constants $\veps$ and $C$ only depend on $N$, $\fkr_\Om$, $\fkd_\Om$, $\fkL$, $\frak c_{\frak 0}$,  $\max_{\ovr\Om} u$ and the $C^{2,\tau}$-regularity
of $\Ga$.
\end{theorem}

\begin{proof}
Thanks to Theorem \ref{th:w bounded usn}, we can repeat the argument of the proof of \cite[Theorem 4.2]{CMS2} in which we replace formula \cite[(3.15)]{CMS2} by \eqref{smallness}. Hence, there exists two concentric balls $B_{r^*_i}$ and $B_{r^*_e}$ and two constants $\ep$ (independent of $\de$) such that 
$$
B_{r^*_i}\subset G \subset B_{r^*_e} \ \mbox{ and } \ r^*_e-r^*_i\le C\,\de^{-1}\, [u]_{\Ga^{\de}},
$$
if $[u]_{\Ga^\de}\leq \ep$.
Moreover, since $\ep$ does not depend on $\de$ and
\begin{equation*}
    \lim_{\de\to 0^+} [u]_{\Ga_\de} = 0,
\end{equation*}
we can find $\de_0 \in (0,\fkr_\Om)$ such that $[u]_{\Ga^\de}\leq \ep$ for any $\de\in (0, \de_0)$.
To complete the proof, we observe that \eqref{Bri-Omega-Bre} and \eqref{stability t}
hold with $r_i=r^*_i+\fkr_\Om/4$ and $r_e=r^*_e+\fkr_\Om/4$.
\end{proof}

\setcounter{equation}{0}
\setcounter{theorem}{0}

\section{Serrin's problem} \label{section Serrin via stability}
In this section, we give a new proof of Serrin's symmetry result and a corresponding stability estimate for spherical symmetry by using the improved stability inequality for the parallel surface problem \eqref{serrin1}, \eqref{parallelcondition}, as just proved in Theorem \ref{thm stability dependence on s}. We need the following lemma.

\begin{lm} \label{lemma seminorm asympt}
Let $\Om$ be a bounded domain with boundary $\Ga$ of class $C^2$ and set $r=\min(\fkr_\Om, \fkR_\Om)$. Let $u$ be
of class $C^2$ in a neighborhood of $\Ga$ and such that $u=0$ on $\Ga$. 
\par
Then
\begin{equation}
\label{seminorms}
[u]_{\Ga^\de}\le \frac{\de}{1-\de/r}\,[u_\nu]_{\Ga}
+\int_0^\de \frac{(\de-t)(r-t)}{r-\de}\,[u_{\nu\nu}]_{\Ga^t}\,dt,
\end{equation}
for every $\de\in [0,r)$.
\par
In particular, if $\de\le r/2$, we have that
\begin{equation}
\label{seminorms2}
[u]_{\Ga^\de}\le 2\,\de\,\left\{[u_\nu]_{\Ga}
+\int_0^\de [u_{\nu\nu}]_{\Ga^t}\,dt\right\}.
\end{equation}
\end{lm}

\begin{proof}
Let $p_1$ and $p_2$ be two points on $\Ga$, so that $p_i^\de=p_i+\de\,\nu(p_i)$, $i=1, 2$,
are points on $\Ga^\de$. It is clear that $|p_1^\de-p_2^\de|\ge |p_1-p_2|-\de\,|\nu(p_1)-\nu(p_2)|$
and hence:
\begin{equation}
\label{p1p2}
|p_1^\de-p_2^\de|\ge  (1-\de/r)\,|p_1-p_2|.
\end{equation}
\par
By applying Taylor's formula to the values of $u$ at $p_1^\de$ and $p_2^\de$ and taking the difference, we have that
\begin{multline*} \label{u k 2}
u(p_1^\de)-u(p_2^\de) = \de\,[u_\nu(p_1)-u_\nu(p_2)] 
+ \int_0^\de (\de-t)\, [u_{\nu\nu}(p_1^t) - u_{\nu\nu}(p_2^t)]\,dt,
\end{multline*}
since $u=0$ at  $p_1$ and $p_2$ and being $\nu(p_i^t)=\nu(p_i)$, $i=1, 2$.
Dividing both sides by $|p_1^\de-p_2^\de|$ and using \eqref{p1p2}, gives that
\begin{multline*}
\frac{|u(p_1^\de)-u(p_2^\de)|}{|p_1^\de-p_2^\de|}\le \frac{\de}{1-\de/r}\,\frac{|u_\nu(p_1)-u_\nu(p_2)|}{|p_1-p_2|}
+\\
\int_0^\de \frac{(\de-t)(r-t)}{r-\de}\,
\frac{|u_{\nu\nu}(p_1^t) - u_{\nu\nu}(p_2^t)|}{|p_1^t-p_2^t |}\,dt\le \\
\frac{\de}{1-\de/r}\,[u_\nu]_{\Ga}
+\int_0^\de \frac{(\de-t)(r-t)}{r-\de}\,[u_{\nu\nu}]_{\Ga^t}\,dt.
\end{multline*}
\par
Therefore, \eqref{seminorms} and hence \eqref{seminorms2} follow at once.
%
\end{proof}

We are now in position to prove both symmetry and stability for Serrin's problem.
Of course, stability implies symmetry, when the normal derivative of $u$ is exactly constant on $\Ga$.
However, we prefer to present the two results separately.

\begin{theorem}[Symmetry] \label{thm Serrin by stability}
Let $\Om\subset\RN$ be a bounded domain with boundary $\Ga$ of class $C^{2,\tau}$, $0<\tau<1$. Let $u\in C^{2,\tau}(\ovr\Om)$,  satisfy \eqref{serrin1} and suppose that \eqref{serrin2} holds with $\fka>0$.
\par
Then $\Om$ is a ball.
\end{theorem}

\begin{proof}
The assumed regularity of $u$ implies that $[u_{\nu\nu}]_{\Ga^t}=O(t^{\tau-1})$ as $t\to 0$
and hence,
since \eqref{serrin2} is in force, \eqref{seminorms2} tells us that \eqref{seminorm-ep} holds for some $\de_0>0$. Thus, Theorem \ref{thm stability dependence on s} can be applied and, by \eqref{seminorms2}, we have that
$$
B_{r_i}\subset \Om \subset B_{r_e} \ \mbox{ and } \ r_e-r_i\le 2 C \int_0^{\de} [u_{\nu\nu}]_{\Ga^t}\,dt
$$
for any $\de\in(0,\de_0)$. The behavior of $[u_{\nu\nu}]_{\Ga^t}$ as $t\to 0$ then implies that the integral at the right-hand side can be made
arbitrarily small. Therefore, $r_e=r_i$, that implies that $\Om$ is a ball.
\end{proof}

\begin{theorem}[Stability] \label{thm stability by stability}
Let $\Om\subset\RN$ be a bounded domain with boundary $\Ga$ of class $C^{2,\tau}$, $0<\tau<1$, and let $u\in C^{2,\tau}(\ovr\Om)$ be solution of \eqref{serrin1}.
Let $C$ be the constant in \eqref{stability t}.
\par
There are two concentric balls $B_{r_i}$ and $B_{r_e}$ such that \eqref{Bri-Omega-Bre}
holds with
\begin{equation} 
\label{stability Serrin}
r_e-r_i\le 2C\, [u_\nu]_{\Ga}.
\end{equation}
\par
In particular, if \eqref{serrin2} is in force with $\fka>0$, then $\Om$ is a ball.
\end{theorem}

\begin{proof}
The regularity of $u$ and \eqref{seminorms2} imply that \eqref{seminorm-ep} holds for some $\de_0>0$. Thus, Theorem \ref{thm stability dependence on s} can be applied and, by \eqref{seminorms2}, we have that \eqref{Bri-Omega-Bre} holds with
$$
r_e-r_i\le 2 C \left\{ [u_\nu]_\Ga+\int_0^{\de} [u_{\nu\nu}]_{\Ga^t}\,dt\right\},
$$
for every $\de\in (0,\de_0)$; \eqref{stability Serrin} then follows by letting $\de$ tend to $0$.
\end{proof}

\begin{remark}
We notice that in Theorem \ref{thm stability by stability} we are not assuming the smallness of $[u_\nu]_{\Ga}$ to prove \eqref{stability Serrin}.
\end{remark}

\section*{Acknowledgements} 
The authors have been supported by the Gruppo Nazionale per l'Analisi Matematica, la Probabilit\`a e le loro
Applicazioni (GNAMPA) of the Istituto Nazionale di Alta Matematica (INdAM). 

The paper was completed while the first author was visiting \lq\lq The Institute for Computational Engineering and Sciences\rq\rq (ICES) of The University of Texas at Austin, and he wishes to thank the institute for the hospitality and support. The first author has been also supported by the NSF-DMS Grant 1361122 and the project FIRB 2013 ``Geometrical and Qualitative aspects of PDE''.

\end{document}